\documentclass[reqno,12pt,letterpaper]{amsart}
\usepackage{amsmath,amssymb,amsthm,graphicx,mathrsfs,url,bbm,array,enumerate}
\usepackage[usenames,dvipsnames]{xcolor}
\usepackage[colorlinks=true,linkcolor=Red,citecolor=Green]{hyperref}
\usepackage{tikz-cd}
\usepackage{tikz}
\usepackage{pgfplots}
\usepackage{subfigure}

\usepackage[
    backend=biber,
    style=alphabetic,
    giveninits=true
]{biblatex}
\DeclareFieldFormat{pages}{#1}
\renewbibmacro{in:}{%
  \ifentrytype{article}
    {}
    {\bibstring{in}%
     \printunit{\intitlepunct}}}
\DeclareFieldFormat
  [article,inbook,incollection,inproceedings,patent,thesis,unpublished]
  {title}{\mkbibemph{#1}}
\DeclareFieldFormat{journaltitle}{#1\isdot}
\DeclareFieldFormat[article]{volume}{\mkbibbold{#1}}
\DeclareFieldFormat[article]{number}{\bibstring{number}\addnbspace #1}

\renewbibmacro*{journal+issuetitle}{%
  \usebibmacro{journal}%
  \setunit*{\addspace}%
  \iffieldundef{series}
    {}
    {\newunit
     \printfield{series}%
     \setunit{\addspace}}%
  \printfield{volume}%
  \setunit{\addspace}%
  \usebibmacro{issue+date}%
  \setunit{\addcomma\space}%
  \printfield{number}%
  \setunit{\addcolon\space}%
  \usebibmacro{issue}%
  \setunit{\addcomma\space}%
  \printfield{eid}
  \newunit}

\newcolumntype{L}{>{$}l<{$}}

\def\?[#1]{\textbf{[#1]}\marginpar{\Large{\textbf{??}}}}
\newtheorem{thm}{Theorem}
\newtheorem{prop}{Proposition}

\newtheorem{lem}[prop]{Lemma}
\newtheorem{cor}[prop]{Corollary}

\numberwithin{equation}{section}
\numberwithin{prop}{section}
\newtheorem*{que}{Question}

\theoremstyle{definition}

\renewcommand{\Re}{\mathop{\rm Re}\nolimits}
\renewcommand{\Im}{\mathop{\rm Im}\nolimits}

\DeclareMathOperator{\Vol}{Vol}

\DeclareMathOperator{\Res}{{\rm Res}}

\newcommand{\Dcal}{{\mathcal D}}

\newcommand{\Ocal}{{\mathcal O}}

\newcommand{\RR}{{\mathbb R}}
\newcommand{\CC}{{\mathbb C}}

\newcommand{\ZZ}{{\mathbb Z}}

\newcommand{\id}{{\rm id}}

\begin{document}
\title[Spectral asymptotics for kinetic Brownian motion]{Spectral asymptotics for kinetic Brownian motion on Riemannian manifolds}

\author{Qiuyu Ren}
\email{qiuyu\_ren@berkeley.edu}
\address{Department of Mathematics, Evans Hall, University of California,
Berkeley, CA 94720, USA}

\author{Zhongkai Tao}
\email{ztao@math.berkeley.edu}
\address{Department of Mathematics, Evans Hall, University of California,
Berkeley, CA 94720, USA}

\begin{abstract}
We prove the convergence of the spectrum of the generator of the kinetic Brownian motion to the spectrum of the base Laplacian for closed Riemannian manifolds. This generalizes recent work of Kolb--Weich--Wolf~\cite{Kolb2022} on constant curvature surfaces and of Ren--Tao \cite{ren2022spectral} on locally symmetric spaces. As an application, we prove a conjecture of Baudoin--Tardif \cite{baudoin2018hypocoercive} on the optimal convergence rate to the equilibrium.
\end{abstract}

\maketitle
\section{Introduction}
Let $(M,g)$ be a closed Riemannian manifold of dimension $n\geq 2$ and $SM=\{(x,v)\in TM: |v|_g=1\}$ be the unit tangent bundle. For any $p\in M$, the fiber $S_pM$ is a standard sphere, so there is a standard (positive) spherical Laplacian $\Delta_{S_pM}$ on $S_pM$. We then define the vertical Laplacian $\Delta_V$ on $SM$ by $(\Delta_Vf)|_{S_pM}:=\Delta_{S_pM}(f|_{S_pM})$ for every $p\in M$. Let $X$ be the generator of the geodesic flow on $SM$. From these two operators we construct the generator of the \textit{kinetic Brownian motion} on $SM$ (see below for motivation) as  
\begin{align}
    P_\gamma:=-\gamma X+c_n\gamma^2\Delta_{V},\qquad c_n=\frac{1}{n(n-1)},\quad \gamma>0.
\end{align}

We are interested in the spectrum of the operator $P_\gamma:D(P_\gamma)=\{u\in L^2(SM): P_\gamma u\in L^2\}\to L^2(SM)$, which we denote by $\sigma(P_\gamma)$. The operator $P_\gamma$ is hypoelliptic, hence it has discrete spectrum with finite multiplicity (see e.g. \cite[Proposition 2.1]{Kolb2022}).
The main result of this paper is 
\begin{thm}\label{thm:Conv}
Let $\Delta_M$ be the
(positive) Laplace--Beltrami operator on $M$, then we have
\begin{align}\label{spec_conv}
    \sigma(P_\gamma)\cap U\to \sigma(\Delta_M)\cap U, \quad \gamma\to\infty
\end{align}
uniformly on any bounded open set $U\Subset \mathbb{C}$, with the agreement of multiplicities. Moreover, for any $s\in\RR$,
\begin{align}\label{res_conv}
    \|(P_\gamma-\lambda)^{-1} - (\Delta_M-\lambda)^{-1}\|_{H^s\to  H^{s+1/4}}\to 0,\quad \gamma\to\infty
\end{align}
uniformly for $\lambda\in U\Subset \CC\setminus \sigma(\Delta_M)$.
\end{thm}


This generalizes the previous work of Kolb--Weich--Wolf ~\cite{KWW2019,Kolb2022} and Ren--Tao \cite{ren2022spectral} to general Riemannian manifolds. The convergence in \eqref{res_conv} is in fact quantitative, and we show in \eqref{eq:res_conv_quant} that the left hand side of \eqref{res_conv} is $\Ocal(\gamma^{-1/10})$. We have not, at this stage, attempted to find the optimal rate of convergence or optimal regularity improvement. This allows us to keep the proof short.

As an application, we prove the following convergence to equilibrium with optimal convergence rate conjectured in Baudoin--Tardif \cite{baudoin2018hypocoercive}.

\begin{thm}\label{thm:Equili}
Suppose in addition to Theorem \ref{thm:Conv} that $M$ is connected and the spectrum of $\Delta_M$ is given by $0=\lambda_0<\lambda_1\leq\lambda_2\leq\cdots$, then for any $0<\beta<\lambda_1$, there is $\gamma_0>0$ such that for any $\gamma>\gamma_0$, there exists $C_\gamma>0$ such that
\begin{align*}
    \left\|e^{-tP_\gamma}u-\frac{1}{\Vol_g(SM)}\int_{SM}u\,  {\rm d Vol}_g\right\|_{L^2}\leq C_\gamma e^{-\beta t}\|u\|_{L^2},\quad t>0.
\end{align*}
\end{thm}
In fact, we have a more precise asymptotic expansion, see Theorem \ref{thm:Expan}.

The operator $P_\gamma$ is the generator of a stochastic process called kinetic Brownian motion. This is a form of a Langevin equation where Brownian motion occurs only in the fiber variables. It was studied by several authors, including Franchi--Le Jan~\cite{FrLe07}, Grothaus--Stilgenbauer ~\cite{GS13}, Angst--Bailleul--Tardif~\cite{ABT15} and Li~\cite{Li16}. It was shown in \cite{ABT15} and \cite{Li16} that kinetic Brownian motion interpolates between the geodesic flow and Brownian motion on the base manifold. Figure \ref{fig:kbm} is a simulation of the kinetic Brownian motion on the flat torus projected to the base. One can see when $\gamma$ is small, it behaves like the geodesic flow; when $\gamma$ is large, it behaves like the Brownian motion on the base manifold.

\begin{figure} 
\centering
\subfigure{
\begin{minipage}[t]{0.32\textwidth} 
\includegraphics[width=1.2\textwidth]{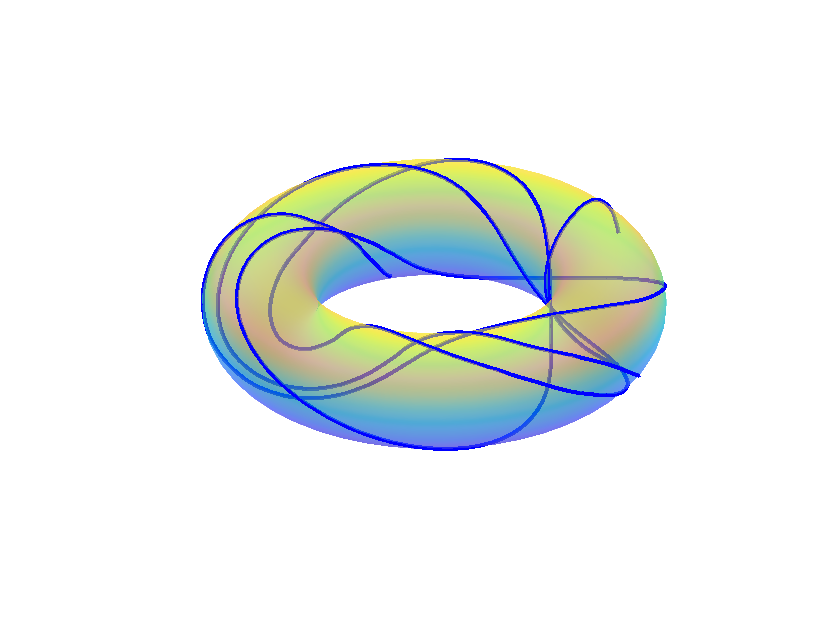} 
\end{minipage}}
\subfigure{\begin{minipage}[t]{0.32\textwidth} 
\includegraphics[width=1.2\textwidth]{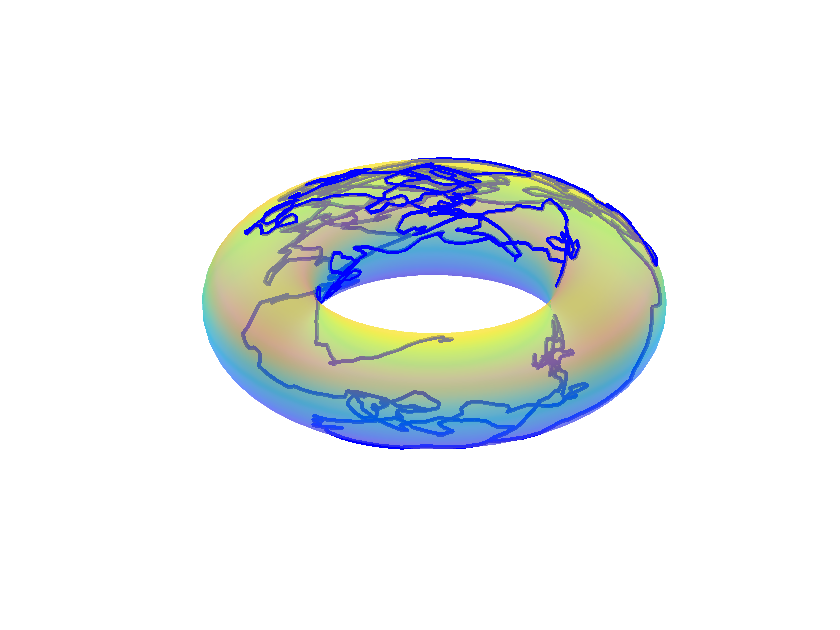} 
\end{minipage}}
\subfigure{\begin{minipage}[t]{0.32\textwidth} 
\includegraphics[width=1.2\textwidth]{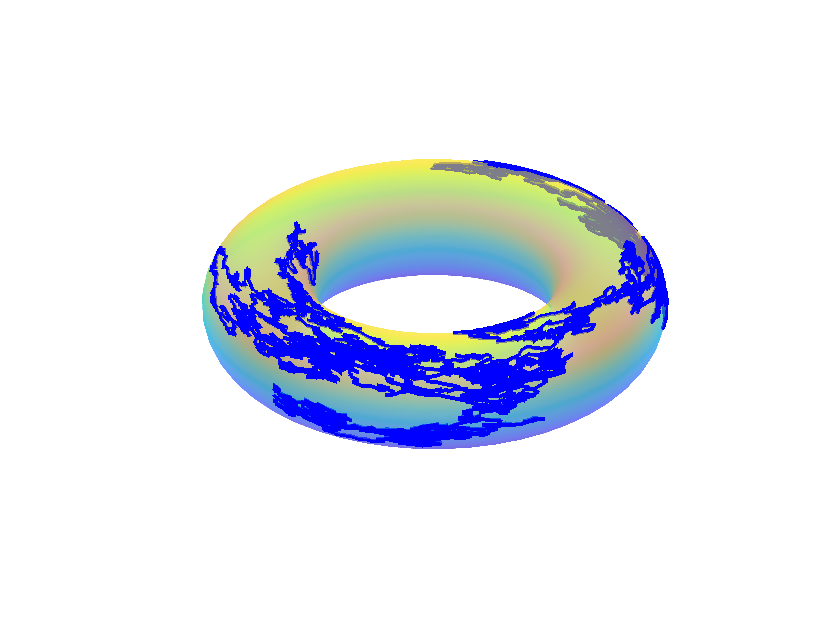} 
\end{minipage}}\caption{Simulation of kinetic Brownian motion for $\gamma=$ \href{https://math.berkeley.edu/~ztao/kbm1.mp4}{$10^{-2}$}, $\href{https://math.berkeley.edu/~ztao/kbm2.mp4}{10}$, $\href{https://math.berkeley.edu/~ztao/kbm3.mp4}{10^4}$ (click for \href{https://math.berkeley.edu/~ztao/kbm_movie.mp4}{movies}) on $S(\RR^2/2\pi\ZZ^2)$ projected to the base $\RR^2/2\pi\ZZ^2$.}
\label{fig:kbm}
\end{figure}

Our motivation also comes from Fried's conjecture, which relates special values of dynamical zeta functions to the Reidemeister torsion -- see Shen \cite{shen2021analytic} for a recent survey. Bismut \cite{bismut2005hypoelliptic} introduced the hypoelliptic Laplacian $L_\gamma$ on the cotangent bundle $T^*M$ as an interpolation between the geodesic flow and Laplacian on the base. Bismut--Lebeau \cite{BL08} proved that $L_\gamma$ converges to $\Delta_M$ in a certain strong sense for arbitrary closed manifolds. Our Theorem \ref{thm:Conv} should be compared with \cite[Theorem 17.21.5]{BL08}. We use a Grushin problem similar to that in \cite{BL08} but we do not use the sophisticated aspects of semiclassical analysis.
This makes our proof simpler and avoids the difficulties due to the fact that $P_\gamma$ is less functorial and does not enjoy good properties coming from the harmonic oscillator structure in the fibers as in \cite[Chapter 16]{BL08}. Bismut \cite{bismut2011hypoelliptic} also studied the limit $\gamma\to 0$ for a related hypoelliptic Laplacian on locally symmetric spaces and obtained formulas for orbit integrals. These lead to the proof of Fried's conjecture in the locally symmetric case in Shen \cite{shen2017analytic}.

One possible reason that it is hard to prove Fried's conjecture using hypoelliptic Laplacian is that we do not know a good convergence of hypoelliptic Laplacian to the geodesic vector field when $\gamma\to 0$ for general negatively curved manifolds. On the other hand, if we think of $P_\gamma$ as an analogue of hypoelliptic Laplacian on $SM$, Drouot \cite{Dr17} proved 
\begin{align*}
    \sigma(X+\gamma\Delta_V)\to \Res(X),\quad\gamma\to 0
\end{align*}
for negatively curved manifolds, where $\Res(X)$ is the spectrum of $X$ on certain anisotropic Sobolev spaces, called Pollicott--Ruelle resonances. Hence it is natural to ask about the limit of $P_\gamma$ as $\gamma\to \infty$. The first breakthrough was achieved by Kolb--Weich--Wolf \cite{KWW2019,Kolb2022}, who proved a weaker version of convergence for constant curvature surfaces. This approach was generalized in Ren--Tao \cite{ren2022spectral} to the case of locally symmetric spaces. We should stress that \cite{ren2022spectral} provides a strong convergence as stated in \eqref{spec_conv}, while \cite{Kolb2022} only proved convergence in  each Casimir eigenspace. The new ingredient in \cite{ren2022spectral} is a careful study of the localization of eigenfunctions in the Fourier space. In this paper, we implement the strategy to general Riemannian manifolds and prove similar localization results. This leads to the proof of Theorem \ref{thm:Conv}. 

The connection between kinetic Brownian motion and Fried's conjecture is still mysterious. We propose it as an open question.
\begin{que}
How is $P_\gamma$ related to the Reidemeister torsion?
\end{que}
A positive answer to this question would give us a new way to understand Fried's conjecture.

The paper is organized as follows. We prove Theorem \ref{thm:Conv} in section \ref{sec2}. This is done by introducing a finite rank absorbing potential $Q_A$ such that $P_\gamma-\lambda+Q_A$ is invertible. In this way, the problem is transformed to the spectrum of the finite rank operator $(P_\gamma-\lambda+Q_A)^{-1}Q_A$, and can be solved using a Grushin problem as in \cite[Section 3.4]{ren2022spectral}. The invertibility of $P_\gamma-\lambda+Q_A$ is proved in Lemma \ref{inv_lem_1} in a quantitative form, and is the key technical difficulty in this paper. Roughly speaking, we study the decomposition into spherical harmonics in the fiber variables and prove that 
eigenfunctions of $P_\gamma$ are localized to $0$-th spherical harmonics. Then we use projection to first order spherical harmonics to conclude eigenfunctions of $P_\gamma$ are also localized in the horizontal direction, hence compleltely localized in the Fourier space. Thus the potential $Q_A$ gives the invertibility of $P_\gamma-\lambda+Q_A$. The improvement of regularity is a corollary of a uniform hypoelliptic estimate (see Proposition \ref{prop:unif_hypo}) following \cite{radkevich1969theorem,kohn1973pseudo, hormander2007analysis}. As an application, we prove Theorem \ref{thm:Equili} in section \ref{sec3}. This is done by writing $e^{-tP_\gamma}$ as the inverse Mellin transform of the resolvent, and then deforming the contour. Our method only provides information on compact sets (or on vertical strips). In the faraway region, we use a result in Eckmann--Hairer \cite{eckmann2003spectral}, which is based on earlier work of H\'erau--Nier \cite{herau2004isotropic}, to obtain a spectral free region near infinity.


\subsection*{Acknowledgement}
We would like to thank Alexis Drouot for sharing with us his notes on kinetic Brownian motion which suggested the Grushin problem used here and in \cite{ren2022spectral}, and for helpful discussions. ZT would also like to thank Maciej Zworski for many helpful discussions and for his encouragement to publish this note. ZT was partially supported by
National Science Foundation under the grant DMS-1901462 and by Simons
Targeted Grant Award No. 896630.

\section{Convergence of spectrum}\label{sec2}
In this section, we prove Theorem \ref{thm:Conv}. We will first recall some important properties of $\Delta_V$ and $X$ studied in Ren--Tao \cite{ren2022spectral}. Then we prove the key invertibility lemmas: Lemma \ref{inv_lem_1} and Lemma \ref{inv_lem_2}, and use them to conclude Theorem \ref{thm:Conv}.

\subsection{Decomposition of $H^s(SM)$}\label{sec:decomp}
The result in this section is basically the same as \cite[Section 3.2]{ren2022spectral}, except we use more general $H^s$ spaces defined as $H^s(SM)=\{u\in \Dcal'(SM):(1+ \Delta)^{s/2}u\in L^2\}$. Here we have three different (positive) Laplacians: the total Laplacian $\Delta$, the horizonal Laplacian $\Delta_H$ and the vertical Laplacian $\Delta_V$.
\begin{itemize}
    \item The total Laplacian $\Delta$ is the Laplace–Beltrami operator associated to the Sasaki metric on $SM$.
    \item The vertical Laplacian $\Delta_V$ is defined as $(\Delta_Vf)|_{S_pM}:=\Delta_{S_pM}(f|_{S_pM})$ for every $p\in M$.
    \item The horizontal Laplacian $\Delta_H$ is defined as $\Delta_H=\Delta-\Delta_V$.
\end{itemize}
We recall from \cite[Theorem 1.5]{bergery1982laplacians} that $\Delta|_{L^2(M)}=\Delta_H|_{L^2(M)}=\Delta_M$ and $\Delta, \Delta_H,\Delta_V$ commute with each other.

Let $s\in\RR$. The total Laplacian $\Delta$ is a self-adjoint operator on $H^s(SM)$ with discrete spectrum. Since $\Delta_V$ commutes with the total Laplacian $\Delta$ on $SM$, we can do spectral decomposition on each eigenspace of $\Delta$. Thus we get the following orthogonal decomposition:
\begin{align}\label{decomposition}
    H^s(SM)=\bigoplus\limits_k V_k^s
\end{align}
where $V_{k}^s=\{u\in H^s(SM): \Delta_V u=k(k+n-2) u\}$ is the $k$-th eigenspace of $\Delta_V$.

Let $\Pi_{k}^s: H^s(SM)\to V_{k}^s$ denote the orthogonal projection with the abbreviated notation $\Pi=\Pi_{0}^s:H^s(SM)\to V_{0}^s$ and $\Pi^\perp=\id-\Pi: H^s(SM)\to V_{>0}^s$. 
The difficulty is that the geodesic vector field
$X$ does not commute with $\Delta_V$, but it satisfies the following properties from \cite[Lemma 3.2]{ren2022spectral}. We include the proofs for completeness.
\begin{lem}
\begin{itemize}
    \item $X$ is anti-self-adjoint with respect to the natural $L^2(SM)$ norm defined via the metric;
    \item $X$ sends $V_{k}$ into $V_{k+1}\oplus V_{k-1}$
with the convention that $V_{-1}=0$;
\item $n\Pi X^2\Pi=-\Delta_M$.
\end{itemize}
\end{lem}
\begin{proof}
\begin{itemize}
    \item The fact $X$ is anti-self-adjoint on $L^2$ follows from the fact that $\exp(tX)$ is volume-preserving. This is essentially Liouville's theorem that geodesic flow preserves the volume.
    \item This is done by a computation in local coordinates. We choose normal coordinates $\{x^i\}$ at $p\in M$, so that $g_{ij}(p)=\delta_{ij}$ and $\partial_kg_{ij}(p)=0$. Then at $p$, $X=\sum v^j\partial_{x^j}$ where $v^j$'s are the induced coordinates on $TM$. The claim follows from the fact that multiplying spherical harmonics of degree $k$ by linear functionals gives a combination of spherical harmonics in degree $k-1$ and $k+1$.
    \item Again we compute in normal coordinates $\{x^i\}$ and the induced coordinates $\{v^i\}$ near the fiber over $p\in M$. The geodesic flow is given by
    \[X=\sum v^i \partial_{x^i}-\sum v^iv^j \Gamma_{ij}^k\partial_{v^k}=\sum(v^i\partial_{x^i}+\mathcal{O}(x)\partial_{v^i}),\]
    and $X^2=\sum v^iv^j\partial_{x^i}\partial_{x^j}+\sum\mathcal{O}(1)\partial_{v^k}+\mathcal{O}(x)$. Since $\partial_{v^k}\Pi=0$, it follows that at $p$, $\Pi X^2\Pi=\sum\Pi(v^iv^j)\partial_{x^i}\partial_{x^j}$.
    Here $\Pi(v^iv^j)$ is the $L^2$ orthogonal projection of $v^iv^j$ to constants. If $i\neq j$, then the projection is zero. If $i=j$, then the projection is given by
    \begin{align*}
       \frac{1}{\Vol(S^{n-1})}\int_{S^{n-1}}(v^j)^2d\sigma(v)=\frac{1}{n\Vol(S^{n-1})}\int_{S^{n-1}}\sum\limits_{i=1}^n(v^j)^2d\sigma(v)=\frac{1}{n}.
    \end{align*}
    Thus $\displaystyle\Pi X^2\Pi=\frac{1}{n}\sum\limits_{i=1}^n\partial^2_{x^i}=-\frac{1}{n}\Delta_M.$\qedhere
\end{itemize}
\end{proof}

\subsection{Invertibility lemmas}
In this section, we prove the crucial invertibility lemmas: Lemma \ref{inv_lem_1} and Lemma \ref{inv_lem_2}. In order to keep track of the dependence on the parameters, we use $A\lesssim_s B$ to mean $A\leq C(s)B$ with the implicit constant $C(s)$ depending on $s$. Similarly, $A\ll_s B$ means we choose $A\leq c(s)B$ for some sufficiently small $c(s)>0$ depending on $s$. Since everything will depend on the dimension $n$ and regularity $s$, we will often omit $s,n$ in the dependence to keep the notations simple.

We start by recalling the hypoelliptic estimate essentially from \cite[Theorem 6.3]{Sm20}.
\begin{lem}
For any $\gamma>0, s\in\RR, N\in\mathbb{N}$, $u\in C^\infty(SM)$ we have
\begin{align}\label{hyposm}
    \|Xu\|_{H^s}+\|\Delta_V u\|_{H^s}+\|u\|_{H^{s+2/3}}\leq C_{\gamma,s,N} (\|P_\gamma u\|_{H^s}+\|u\|_{H^{-N}}).
\end{align}
\end{lem}
Since $\Delta$ commutes with $\Delta_V$, it is easy to see $C^\infty(SM)$ is dense in $D^s(P_\gamma)=\{u\in H^s(SM):P_\gamma u\in H^s(SM)\}$. Thus \eqref{hyposm} works for any $u\in D^s(P_\gamma)$. In particular, $(P_\gamma-\lambda)u\in C^\infty(SM)$ implies $u\in  C^\infty(SM)$. A basic accretive estimate shows $P_\gamma:D^s(P_\gamma)\to H^s(SM)$ is a Fredholm operator with index $0$.
\begin{lem}
For $\Re\lambda< 0$, $P_\gamma-\lambda$ is invertible on $L^2$. For $\Re\lambda<0$ sufficiently negative (depending on $s$ and $\gamma$), $P_\gamma-\lambda$ is invertible on $H^s$.
\end{lem}
\begin{proof}
We will only prove the claim for $L^2$. The proof for $H^s$ is similar. First we recall $\Re(Xu,u)_{L^2}=0$ since $X$ is anti-self-adjoint. Thus
\begin{align*}
    \Re((P_\gamma-\lambda)u,u)=c_n\gamma^2(\Delta_V u,u)-\Re\lambda \|u\|^2\geq -\Re\lambda \|u\|^2.
\end{align*}
For $\Re\lambda<0$, this shows $P_\gamma-\lambda:D(P_\gamma)\to L^2$ is injective and the image is closed. We claim it is also surjective. If there is $v\in L^2(SM)$ such that $((P_\gamma-\lambda)u,v)=0$, then distributionally
\begin{align*}
    (P_\gamma^*-\bar{\lambda})v=0.
\end{align*}
By hypoellipticity, $v\in C^\infty(SM)$. However,
\begin{align*}
    0=\Re((P_\gamma^*-\bar{\lambda})v,v)=\gamma^2(\Delta_V v,v)-\Re\lambda\|v\|^2\geq -\Re\lambda\|v\|^2
\end{align*}
implies $v=0$. So $P_\gamma$ must be surjective and thus invertible.
\end{proof}

When $\Re\lambda\geq 0$, it is possible that $P_\gamma-\lambda$ is not invertible. The following lemma essentially says that any such eigenfunction must be localized to finite frequency. In order to implement the heuristics, for $A>0$ we introduce $Q_A=A^2\Pi \mathbbm{1}_{(\Delta_M\leq A^2)}\Pi:H^s(SM)\to H^s(SM)$. This is a finite rank smoothing operator localized to finite frequencies. Here $\mathbbm{1}_{(\lambda\leq A^2)}$ is the characteristic function of the set $\{\lambda\leq A^2\}$ and $\mathbbm{1}_{(\Delta_M\leq A^2)}$ is the spectral projection to eigenspaces of $\Delta_M$ with eigenvalue $\leq A^2$, defined using functional calculus of self-adjoint operators.

\begin{lem}\label{inv_lem_1}
For any $C_0>0$, $s\in \RR$, there exists $C_1=C_1(C_0,n)>0$ such that for any $\gamma>C_1$, $A>C_1$ and $|\lambda|\leq C_0$, the operator \begin{align*}
    P_\gamma-\lambda+Q_A:D^s(P_\gamma)=\{u\in H^s(SM): P_\gamma u\in H^s\}\to H^s(SM)
\end{align*}
is invertible. For $\gamma>A\gg_{C_0,n,s}1$, the inverse has the bound
\begin{align}\label{eq:inv_bound_Hs}
    \|(P_\gamma-\lambda+Q_A)^{-1}\|_{H^{s}\to H^s}\lesssim_{C_0,n,s}A^{-1}.
\end{align}
\end{lem}
\begin{proof}

Since $P_\gamma-\lambda+Q_A$ is hypoelliptic and Fredholm of index $0$, we only need to prove it has no kernel. Suppose by contradiction that for some $u\in H^s(SM)\setminus\{0\}$,
\begin{align}\label{inv_eqn}
    (P_\gamma-\lambda+Q_A)u=0.
\end{align}
Then $u\in C^\infty$ by hypoellipticity. Suppose $\|u\|_{L^2}=1$ and denote $u_k=\Pi_k u$. Pairing with $u$ gives
\begin{align}\label{inv_eqn_L2}
    c_n\gamma^2(\Delta_V u,u)-\gamma\Re(Xu,u)-\Re\lambda \|u\|_{L^2}^2+(Q_A u_0,u_0)=0.
\end{align}
Since $\Re(Xu,u)_{L^2}=0$, we get
\begin{align*}
    \|\Pi^\perp u\|_{L^2}\lesssim_{C_0} \gamma^{-1}.
\end{align*}
Similaly pairing \eqref{inv_eqn} with $(\Delta_H+1) u$ gives
\begin{align*}
    c_n\gamma^2(\Delta_V(\Delta_H+1) u,u)-\gamma\Re (Xu, (\Delta_H+1) u)-\Re\lambda\|u\|_{H^1}^2+(Q_A u_0,(\Delta_H+1)u_0)=0. 
\end{align*}
Moreover, 
\begin{align*}
    2\Re(Xu,(\Delta_H+1) u)=([\Delta_H,X](\Pi^\perp u+u_0),\Pi^\perp u+u_0)\lesssim \|\Pi^\perp u\|_{H^1}^2+\|\Pi^\perp u\|_{H^1}\|u_0\|_{H^1}.
\end{align*}
Note $\|\Pi^\perp u\|_{H^1}\|u_0\|_{H^1}\leq \epsilon_n \gamma\|\Pi^\perp u\|_{H^1}^2+\gamma^{-1}\epsilon_n^{-1}\|u_0\|_{H^1}^2$, we conclude
\begin{align*}
    \|\Pi^\perp u\|_{H^1}\lesssim_{C_0} \gamma^{-1}\| u_0\|_{H^1}.
\end{align*}
We come back to \eqref{inv_eqn}. Projecting it to $V_1$ gives
\begin{align*}
    \frac{1}{n}\gamma^2 u_1-\gamma \Pi_1 X (u_0+u_2)=\lambda u_1.
\end{align*} 
Recall $\|u_1\|_{L^2}\lesssim_{C_0}\gamma^{-1}$, and
\begin{align*}
    \|Xu_0\|_{L^2}^2&=-(\Pi X^2\Pi u_0,u_0)=\frac{1}{n}(\Delta_M u_0,u_0)\gtrsim\|u_0\|_{H^1}^2-\|u_0\|_{L^2}^2,\\
    \|Xu_2\|_{L^2}^2&\lesssim \|u_2\|_{H^1}^2\lesssim_{C_0}\gamma^{-1} \|u_0\|_{H^1}^2. 
\end{align*}
We conclude $\|u_0\|_{H^1}\lesssim_{C_0} 1$. The key observation is that the constant in this estimate is independent of $A$.
On the other hand, \eqref{inv_eqn_L2} also implies $ \|A\mathbbm{1}_{(\Delta\leq A^2)}u_0\|_{L^2}\lesssim_{C_0} 1$.
Taking $A\gg_{C_0} 1$  gives a contradiction. This shows the invertibility of $P_\gamma-\lambda+Q_A$.

In order to get the bound for the inverse, let $f\in C^\infty$ and $u\in C^\infty$ such that
\begin{align}\label{inv_eqn_f}
    (P_\gamma-\lambda+Q_A)u=f.
\end{align}
Pairing with $u$ in $H^s$ gives
\begin{align}\label{inv_eqn_Hs}
    c_n\gamma^2(\Delta_V u,u)_{H^s}-\gamma\Re(Xu,u)_{H^s}-\Re\lambda\|u\|_{H^s}^2+(Q_Au,u)_{H^s}=(f,u)_{H^s}.
\end{align}
Since
\begin{align*}
    \Re(Xu,u)_{H^s}=\Re(Xu,(1+\Delta)^su)_{L^2}&=\frac{1}{2}([(1+\Delta)^s,X]u,u)_{L^2}\\
    &\lesssim \|\Pi^\perp u\|_{H^{s}}^2+\|\Pi^\perp u\|_{H^s}\|u_0\|_{H^s},
\end{align*}
we conclude as before
\begin{align*}
    \|\Pi^\perp u\|_{H^s}\lesssim_{C_0}\gamma^{-1}(\|u_0\|_{H^s}+\|f\|_{H^s}),\quad     \|A\mathbbm{1}_{(\Delta\leq A^2)} u_0\|_{H^s}\lesssim_{C_0}\|u_0\|_{H^s}+\|f\|_{H^s}.
\end{align*}


Now we look at the $\Pi_1$ component of \eqref{inv_eqn_f}, i.e.
\begin{align*}
        \frac{1}{n}\gamma^2 u_1-\gamma \Pi_1 X (u_0+u_2)=\lambda u_1+f_1.
\end{align*}
We conclude
\begin{equation}\label{eq: Pi_1 bound}
\begin{aligned}
    \|u_0\|_{H^s}&\lesssim\|Xu_0\|_{H^{s-1}}+\|u_0\|_{H^{s-1}}\\
    &\lesssim_{C_0} \gamma\|u_1\|_{H^{s-1}}+\gamma^{-1}\|f\|_{H^{s-1}}+\|u_2\|_{H^s}+\|u_0\|_{H^{s-1}}\\
    &\lesssim_{C_0}\|u_0\|_{H^{s-1}}+\|f\|_{H^{s-1}}+\gamma^{-1}\|f\|_{H^s}.
\end{aligned}
\end{equation}
Now we divide into two cases. 
\begin{itemize}
\item If $f=\mathbbm{1}_{(\Delta>\gamma^{2})}f$ is in high frequency, then \eqref{eq: Pi_1 bound} implies that
\begin{align*}
\|u_0\|_{H^s}&\lesssim_{C_0}\|\mathbbm{1}_{(\Delta\leq A^2)} u_0\|_{H^{s-1}}+\|\mathbbm{1}_{(\Delta>A^2)}u_0\|_{H^{s-1}}+\gamma^{-1}\|f\|_{H^s}\\
&\lesssim_{C_0} A^{-1}(\|u_0\|_{H^{s-1}}+\|f\|_{H^{s-1}})+A^{-1}\|u_0\|_{H^s}+\gamma^{-1}\|f\|_{H^s}\\
&\lesssim A^{-1}\|u_0\|_{H^s}+A^{-1}\|f\|_{H^s}.
\end{align*}

\item If $f=\mathbbm{1}_{(\Delta\leq \gamma^2)}f$ is in low frequency, then \eqref{eq: Pi_1 bound} with $s$ replaced by $s+1$ gives
\begin{align*}
\|u_0\|_{H^{s+1}}&\lesssim_{C_0}\|u_0\|_{H^s}+\|f\|_{H^s}+\gamma^{-1}\|f\|_{H^{s+1}}\\
&\lesssim \|u_0\|_{H^s}+\|f\|_{H^s}.
\end{align*}
On the other hand, we have
\begin{align*}
\|u_0\|_{H^{s+1}}\geq\|\mathbbm{1}_{(\Delta>A^2)}u_0\|_{H^{s+1}}\geq A\|\mathbbm{1}_{(\Delta>A^2)}u_0\|_{H^s},
\end{align*}
thus
\begin{align*}
\|u_0\|_{H^s}&\leq \|\mathbbm{1}_{(\Delta\leq A^2)}u_0\|_{H^s}+\|\mathbbm{1}_{(\Delta>A^2)}u_0\|_{H^s}\\
    &\lesssim_{C_0} A^{-1}(\|u_0\|_{H^s}+\|f\|_{H^s}).
\end{align*}

\end{itemize}
In both cases we have $\|u_0\|_{H^s}\lesssim_{C_0} A^{-1}\|f\|_{H^s}$ and we conclude 
\[\|u\|_{H^s}\lesssim_{C_0}\|u_0\|_{H^s}+\gamma^{-1}\|f\|_{H^s}\lesssim_{C_0}A^{-1}\|f\|_{H^s}. \qedhere\]
\end{proof}

In order to get the improvement of regularity in \eqref{res_conv}, we prove a uniform hypoelliptic estimate following \cite{radkevich1969theorem, kohn1973pseudo} and \cite[Theorem 22.2.1]{hormander2007analysis}.
\begin{prop}\label{prop:unif_hypo}
For $A,B>1$, $\gamma>A+B^2$, $y\in\RR$, there exists $C=C(n,s)$ independent of $A,B, \gamma, y$ such that 
\begin{align}\label{eq:unif_hypo_QA}
    \|u\|_{H^{s+1/4}}\leq CB^{-1}\|(P_\gamma+Q_A)u\|_{H^s}+CB\|u\|_{H^s}.
\end{align}
\begin{align}\label{eq:unif_hypo_iy}
    \|u\|_{H^{s+1/8}}\leq CB^{-1}\|(P_\gamma - iy) u\|_{H^s}+CB\|u\|_{H^s}.
\end{align}
\end{prop}
\begin{proof}
It suffices to compute locally.
We will only give the proof of \eqref{eq:unif_hypo_QA}, but \eqref{eq:unif_hypo_iy} is proved in the same way using the fact that for a local basis $X_i$ of vertical vector fields, the vector fields $$X_i, [X_i,X], [X_i,[X_i, X]], i=1,2,\cdots,n-1$$ generate all directions. 

In order to get \eqref{eq:unif_hypo_QA}, since $X, X_i, [X_i,X], i=1,2,\cdots, n-1$ generate all directions, it suffices to bound $\|X_iu\|_{H^{s}}$, $\|Xu\|_{H^{s-1/2}}$ and $\|[X_i,X]u\|_{H^{s-3/4}}$ by the right hand side of \eqref{eq:unif_hypo_QA}. First, we have 
\begin{equation}\label{eq:unif_hypo_QA1}
     \begin{aligned}
     (Q_Au,u)_{H^s}+\gamma^2(\Delta_V u, u)_{H^s}&\lesssim \Re((P_\gamma+Q_A)u,u)_{H^s}+C\|u\|_{H^s}^2\\
    &\lesssim B^{-2}\|(P_\gamma+Q_A)u\|_{H^s}^2+B^2\|u\|_{H^s}^2.
       \end{aligned} 
\end{equation}

We will abbreviate pseudodifferential operators of order $k$ by $\Psi^k$ to simplify the notation. For $\|Xu\|_{H^{s-1/2}}$, we have
\begin{align*}
    \|Xu\|_{H^{s-1/2}}^2=(Xu,\Psi^0 u)_{H^s}=\gamma^{-1}((c_n\gamma^2\Delta_V+Q_A-(P_\gamma+Q_A))u,\Psi^0 u)_{H^s}.
\end{align*}
The first term is estimated as
\begin{align*}
    \gamma^{-1}(\gamma^2\Delta_V u,\Psi^0 u)_{H^s}&=\gamma^{-1}(\gamma \nabla^V u,\Psi^0\gamma \nabla^V u)_{H^s}+(\gamma\nabla^Vu,\Psi^0 u)_{H^s}\\
    &\lesssim B^{-2}\|(P_\gamma+Q_A)u\|_{H^s}^2+B^2\|u\|_{H^s}^2.
\end{align*}
The second term is estimated as
\begin{align*}
    \gamma^{-1}(Q_Au, \Psi^0 u)_{H^s}\lesssim \gamma^{-1}\|Q_Au\|_{H^s}\|u\|_{H^s}\lesssim B^{-2}\|(P_\gamma+Q_A)u\|_{H^s}^2+B^2\|u\|_{H^s}^2.
\end{align*}
The third term is estimated as
\begin{align*}
    \gamma^{-1}((P_\gamma+Q_A)u,\Psi^0 u)_{H^s}\lesssim \gamma^{-1}\|(P_\gamma+Q_A)u\|_{H^s}\|u\|_{H^s}\leq \gamma^{-1}\|(P_\gamma+Q_A)u\|_{H^s}^2+\gamma^{-1}\|u\|_{H^s}^2.
\end{align*}
Thus we conclude
\begin{align}\label{eq:unif_hypo_QA2}
    \|Xu\|_{H^{s-1/2}}\lesssim B^{-1}\|(P_\gamma+Q_A)u\|_{H^s}+B\|u\|_{H^s}.
\end{align}

For $\|[X_i,X]u\|_{H^{s-3/4}}$, we have 
\begin{align*}
&\|[X_i,X]u\|_{H^{s-3/4}}^2\\&=([X_i,X]u, \Psi^{-1/2} u)_{H^s}\\
&=(X_iXu,\Psi^{-1/2} u)_{H^s}-(X X_i u,\Psi^{-1/2} u)_{H^s}\\
&=-(X u, \Psi^{-1/2} X_i u)_{H^s}+(X u, \Psi^{-1/2} u)_{H^s}+ (X_i u, \Psi^{-1/2} X u)_{H^s}-(X_i u, \Psi^{-1/2} u)_{H^s}\\
&=\Re (X u, \Psi^{-1/2} X_i u)_{H^s} +\Re(X u, \Psi^{-1/2} u )_{H^s} +\Re(X_iu, \Psi^{-1/2}u)_{H^s}.
\end{align*}
The last term is estimated as
\begin{align*}
    \Re(X_i u, \Psi^{-1/2}u)_{H^s}\lesssim \gamma^{-1}\|(P_\gamma+Q_A) u\|_{H^s}^2+\gamma^{-1}\|u\|_{H^s}^2,
\end{align*}
The second term is estimated as
\begin{align*}
    &\Re(Xu, \Psi^{-1/2} u )_{H^s}
    \\&=\gamma^{-1}\Re((c_n\gamma^2\Delta_V+Q_A-(P_\gamma+Q_A) )u,\Psi^{-1/2}u)_{H^s}\\
    &\lesssim \gamma^{-1}\|\gamma\nabla^V u\|_{H^s}\|\gamma\nabla^V \Psi^{-1/2}u\|_{H^s}+\gamma^{-1}(Q_A u, \Psi^{-1/2}u)_{H^s}+ \gamma^{-1}\|(P_\gamma+Q_A)u\|_{H^s}^2+\gamma^{-1}\|u\|_{H^s}^2\\
    &\lesssim B^{-2}\|(P_\gamma+Q_A)u\|_{H^s}^2+B^2\|u\|_{H^s}^2.
\end{align*}
The first term is estimated as
\begin{align*}
    &\Re (X u, \Psi^{-1/2} X_i u)_{H^s}=\gamma^{-1}\Re((c_n\gamma^2\Delta_V+Q_A- (P_\gamma+Q_A) )u,\Psi^{-1/2}X_iu)_{H^s}\\
    &\lesssim \gamma^{-1}(\gamma^2\Delta_V u,u)_{H^s}^{1/2}(\gamma^2\Delta_V \Psi^{-1/2}X_i u,\Psi^{-1/2}X_iu)_{H^s}^{1/2}\\
    &+\gamma^{-1}(Q_A u, \Psi^{-1/2}X_i u)_{H^s} +\gamma^{-1}\|(P_\gamma+Q_A)u\|_{H^s}\|X_iu\|_{H^s}\\
    &\lesssim  B^{-2}\|(P_\gamma+Q_A)u\|_{H^s}^2+B^2\|u\|_{H^s}^2
\end{align*}
where we used $(\Delta_V u,v)_{H^s}\leq (\Delta_V u,u)_{H^s}^{1/2}(\Delta_V v,v)_{H^s}^{1/2}$ and
\begin{align*}
    &(\gamma^2\Delta_V \Psi^{-1/2}X_i u,\Psi^{-1/2}X_iu)_{H^s}\\
    &\lesssim \Re(P_\gamma\Psi^{-1/2}X_i u,\Psi^{-1/2}X_i u)_{H^s}+ C\|\Psi^{-1/2}X_i u\|_{H^s}^2\\
    &\lesssim \Re(\Psi^{-1/2}X_i P_\gamma u,\Psi^{-1/2}X_i u)_{H^s}+\Re((\gamma^2\Psi^{1/2}\nabla^V+\gamma \Psi^{1/2}) u,\Psi^{-1/2}X_iu)_{H^s}+C\|\Psi^{-1/2}X_i u\|_{H^s}^2\\
    &\lesssim B^{-2}\|P_\gamma u\|_{H^s}^2+B^2\|u\|_{H^s}^2\\
    &\lesssim B^{-2}\|(P_\gamma+Q_A) u\|_{H^s}^2+B^2\|u\|_{H^s}^2+\|Q_Au\|_{H^s}^2.
\end{align*}

Thus we conclude
\begin{align}\label{eq:unif_hypo_QA3}
    \|[X_i,X]u\|_{H^{s-3/4}}\lesssim B^{-1}\|(P_\gamma+Q_A)u\|_{H^s}+B\|u\|_{H^s}.
\end{align}

Combining \eqref{eq:unif_hypo_QA1}, \eqref{eq:unif_hypo_QA2}, \eqref{eq:unif_hypo_QA3}, we conclude \eqref{eq:unif_hypo_QA}.
\end{proof}

As a corollary, we can improve the regularity in \eqref{eq:inv_bound_Hs}.
\begin{cor}
In Lemma \ref{inv_lem_1}, we have
\begin{align}\label{eq:inv_bound_1/4}
    \|(P_\gamma-\lambda+Q_A)^{-1}\|_{H^s\to H^{s+1/4}}\lesssim_{C_0,n,s} A^{-1/2}.
\end{align}
\end{cor}
\begin{proof}
We take $B=A^{1/2}$ in \eqref{eq:unif_hypo_QA}, then
\begin{align*}
    \|u\|_{H^{s+1/4}}&\lesssim_{C_0} A^{-1/2}\|(P_\gamma-\lambda+Q_A)u\|_{H^s}+A^{1/2}\|u\|_{H^s}\\
    &\lesssim_{C_0} A^{-1/2}\|(P_\gamma-\lambda+Q_A)u\|_{H^s}.\qedhere
\end{align*}
\end{proof}

We will also need the following invertibility lemma. We use the semiclassical notation $h=\gamma^{-1}$ and $\tilde{P}_h=c_n\Delta_V-hX$. Note that $P_\gamma=\gamma^2\tilde{P}_h$. 
\begin{lem}\label{inv_lem_2}
Let $s\in\RR$, $|\lambda|\leq C_0$, there exists $h_0=h_0(C_0,n,s)>0$ such that for $0<h<h_0$, the operator
\begin{align*}
    \Pi^\perp (\tilde{P}_h-h^2\lambda)\Pi^\perp:\{u\in V_{>0}^s: \Pi^\perp\tilde{P}_h u\in H^s\}\to V_{>0}^s
\end{align*} is invertible. The inverse has norm
\begin{align*}
    \|( \Pi^\perp (\tilde{P}_h-h^2\lambda)\Pi^\perp)^{-1}\|_{H^s\to H^s}\lesssim_{C_0,n,s}1.
\end{align*}
\end{lem}
\begin{proof}
For $u\in V_{>0}^s$, $h\ll_{C_0} 1$,
\begin{align*}
    \Re((\tilde{P}_h-h^2\lambda)u,u)_{H^s}=c_n(\Delta_V u,u)_{H^s}-h\Re(Xu,u)_{H^s}-h^2\Re\lambda\|u\|_{H^s}^2\gtrsim_{C_0} \|u\|_{H^s}^2.
\end{align*}
So $\Pi^\perp (\tilde{P}_h-h^2\lambda)\Pi^\perp$ is injective and has closed image. Suppose it is not surjective, then there exists a nonzero $v\in V_{>0}^s$ such that \begin{align*}
    (\Pi^\perp(\tilde{P}_h-h^2\lambda)u,v)_{H^s}=0,\quad\forall u\in C^\infty(SM)\cap V_{>0}.
\end{align*}
Thus $\Pi^\perp(\tilde{P}_h^*-h^2\bar{\lambda})v=0$ where the adjoint is taken in $H^s$.
Let $\chi\in C_0^\infty(\RR;[0,1])$ be a cutoff function such that $\chi=1$ near $0$ and $v_\epsilon=\chi(\epsilon^2\Delta)v$, then
\begin{align*}
\|v_\epsilon\|_{H^s}&\lesssim_{C_0}     \|\Pi^\perp(\tilde{P}_h^*-h^2\bar{\lambda})v_\epsilon\|_{H^s}\\
&\lesssim_{C_0} h\|[ X^*,\chi(\epsilon^2 \Delta)]v\|_{H^s}\\
&\lesssim_{C_0} h\|v\|_{H^s}.
\end{align*}
Let $h\ll_{C_0} 1$ and $\epsilon\ll 1$, we conclude $v=0$, a contradiction. Thus $\Pi^\perp (\tilde{P}_h-h^2\lambda)\Pi^\perp$ is also surjective and thus invertible.
\end{proof}

\subsection{Spectral convergence}
In this section we prove the convergence of the spectrum in Theorem~\ref{thm:Conv} by a Grushin problem following \cite{ren2022spectral}.

Let $i_0: V_{0}^s\to H^s(SM)$ be the inclusion. Intuitively, we want to consider the following Grushin problem for $P_\gamma-\lambda+Q_A$.
\begin{align*}
    \begin{pmatrix}
    P_\gamma-\lambda+Q_A &  \gamma i_0\\
    \gamma\Pi&0
    \end{pmatrix}:
    D^s(P_\gamma)  \oplus  V_{0}^s\to H^s(SM)\oplus V_{0}^s.
\end{align*}
However, it is not clear what the correct space is to set up the Grushin problem. Instead we will just directly write down a formula \eqref{P_gamma_inv_grushin} that works distributionally. Using same methods in \cite{ren2022spectral}, we can solve the equations
\begin{align}\label{equ: grushin}
    \left\{\begin{array}{ll}
    (P_\gamma-\lambda+Q_A) u+\gamma u_- =v,  \quad\quad  &(u,u_-)\in \Dcal'(SM)\oplus \Dcal'(M),  \\
    \gamma\Pi u=v_+,     & (v,v_+)\in \Dcal'(SM)\oplus \Dcal'(M).
    \end{array}\right.
\end{align}
The solution we get is
\begin{align}\label{sol:grushin}
    \left\{\begin{array}{cl}
    u&=(\Pi^\perp (P_\gamma-\lambda)\Pi^\perp)^{-1}\Pi^\perp (v+X v_+)+\gamma^{-1} v_+,    \\
    u_-&= \gamma^{-1}\Pi v+\gamma^{-2}(\lambda-Q_A) v_++\Pi X(\Pi^\perp (P_\gamma-\lambda)\Pi^\perp)^{-1}\Pi^\perp (v+Xv_+).
    \end{array}\right.
\end{align}

Now we write (at least formally)
\begin{align}\label{P_gamma_inv_grushin}
    (P_\gamma-\lambda+Q_A)^{-1}=E-E_+E_{-+}^{-1}E_-
\end{align}
where
\begin{align*}
    E&=(\Pi^\perp (P_\gamma-\lambda)\Pi^\perp)^{-1}\Pi^\perp,
    &E_+&=(\Pi^\perp (P_\gamma-\lambda)\Pi^\perp)^{-1}\Pi^\perp X+\gamma^{-1},\\
   E_-&=\gamma^{-1}\Pi+\Pi X(\Pi^\perp (P_\gamma-\lambda)\Pi^\perp)^{-1}\Pi^\perp,
   &E_{-+}&=\gamma^{-2}( \lambda +\Pi X(\Pi^\perp (\tilde{P}_h-h^2\lambda)\Pi^\perp)^{-1}X\Pi-Q_A).
\end{align*}
We need to justify that $E_{-+}$ is invertible, and the inverse has a good control. So let us look at the equation $\gamma^2 E_{-+}u=f$. Let $v=(\Pi^\perp (\tilde{P}_h-h^2\lambda)\Pi^\perp)^{-1}X\Pi u$, we have
\begin{align*}
   \lambda u +\Pi X v -Q_A u=f,\qquad \Pi^\perp(\tilde{P}_h-h^2\lambda)v=X u. 
\end{align*}
Thus $\gamma^2 E_{-+}u=f$ is equivalent to
\begin{align*}
    (P_\gamma-\lambda+Q_A)(u+hv)=-f.
\end{align*}
By Lemma \ref{inv_lem_1}, $P_\gamma-\lambda+Q_A$ is invertible, we conclude $E_{-+}$ is also invertible. So the formula \eqref{P_gamma_inv_grushin} makes sense distributionally for $\gamma>A\gg_{C_0,n,s} 1$ depending on the Sobolev regularity $s$. 

In order to apply \eqref{P_gamma_inv_grushin}, we write 
\begin{align*}
    P_\gamma-\lambda=P_\gamma-\lambda+Q_A-Q_A=(P_\gamma-\lambda+Q_A)(I-(P_\gamma-\lambda+Q_A)^{-1}Q_A).
\end{align*}
We claim
\begin{prop}\label{prop:conv_QA}
For $|\lambda|\leq C_0$, $\gamma>A\gg_{C_0,n,s,N}1$, we have
\begin{align*}
    \|(P_\gamma-\lambda+Q_A)^{-1}Q_A - (\Delta_M-\lambda+Q_A)^{-1}Q_A\|_{H^s\to H^{s+N}}\lesssim_{C_0,n,s,N}A^{N+2}\gamma^{-1},\\
    \|Q_A(P_\gamma-\lambda+Q_A)^{-1} - Q_A(\Delta_M-\lambda+Q_A)^{-1}\|_{H^s\to H^{s+N}}\lesssim_{C_0,n,s,N}A^{N+2}\gamma^{-1},
\end{align*}
for any $s\in\RR, N\geq 0$.
\end{prop}
\begin{proof}
We will only prove the first one, but the second one is proved exactly the same way.

Note by \eqref{P_gamma_inv_grushin}, $ (P_\gamma-\lambda+Q_A)^{-1}Q_A=-\gamma^{-1}E_+E_{-+}^{-1}Q_A$. We first prove a bound for $\gamma^{-2} E_{-+}^{-1}Q_A$ in $H^s\to H^{s+N}$ for any $N\geq 0$. Let $\gamma^2 E_{-+} u=Q_Af$, then for $v=(\Pi^\perp (\tilde{P}_h-h^2\lambda)\Pi^\perp)^{-1}X\Pi u$ we have
\begin{align*}
    (P_\gamma-\lambda+Q_A)(u+hv)=-Q_Af.
\end{align*}
By Lemma \ref{inv_lem_1}, we conclude
\begin{align*}
    \|u\|_{H^{s+N}}\lesssim_{C_0}A^{-1}\|Q_A f\|_{H^{s+N}}\lesssim_{C_0}A^{N+1}\|f\|_{H^s}.
\end{align*}
Now we can estimate the difference
\begin{align*}
    &(P_\gamma-\lambda+Q_A)^{-1}Q_A- (\Delta_M-\lambda+Q_A)^{-1}Q_A=-\gamma^{-1}E_+E_{-+}^{-1}Q_A-(\Delta_M-\lambda+Q_A)^{-1}Q_A\\
    &=(-\gamma^{-2}E_{-+}^{-1}-(\Delta_M-\lambda+Q_A)^{-1})Q_A-\gamma^{-3}(\Pi^\perp(\tilde{P}_h-h^2\lambda)\Pi^\perp)^{-1}\Pi^\perp X E_{-+}^{-1}Q_A.
\end{align*}
For the second term,
\begin{align*}
    \|\gamma^{-3}(\Pi^\perp(\tilde{P}_h-h^2\lambda)\Pi^\perp)^{-1}\Pi^\perp X E_{-+}^{-1}Q_A\|_{H^s\to H^{s+N}}&\lesssim_{C_0}\|\gamma^{-3}E_{-+}^{-1}Q_A\|_{H^s\to H^{s+N+1}}\\
    &\lesssim_{C_0}A^{N+2}\gamma^{-1}.
\end{align*}
For the first term,
\begin{align*}
    &(-\gamma^{-2}E_{-+}^{-1}-(\Delta_M-\lambda+Q_A)^{-1})Q_A\\
    =\,& (\Delta_M-\lambda+Q_A)^{-1}(-\Delta_M-\Pi X (\Pi^\perp(\tilde{P}_h-h^2\lambda)\Pi^\perp)^{-1}X\Pi)\gamma^{-2}E_{-+}^{-1}Q_A\\
    =\,&(\Delta_M-\lambda+Q_A)^{-1}(\Pi X ((c_n\Pi^\perp\Delta_V\Pi^\perp)^{-1}-(\Pi^\perp(\tilde{P}_h-h^2\lambda)\Pi^\perp)^{-1})X\Pi)\gamma^{-2}E_{-+}^{-1}Q_A\\
    =\,&(\Delta_M-\lambda+Q_A)^{-1}(\Pi X ((c_n\Pi^\perp\Delta_V\Pi^\perp)^{-1}\Pi^\perp(-hX-h^2\lambda)\Pi^\perp\\
    &\qquad \qquad\qquad \qquad\qquad \qquad\qquad \qquad\qquad  (\Pi^\perp(\tilde{P}_h-h^2\lambda)\Pi^\perp)^{-1})X\Pi)\gamma^{-2}E_{-+}^{-1}Q_A.
\end{align*}
Note
\begin{align*}
    \|\gamma^{-2}E_{-+}^{-1}Q_A\|_{H^s\to H^{s+N+3}}\lesssim_{C_0}A^{N+4},\|X\|_{H^{s+N+3}\to H^{s+N+2}}\lesssim 1,
\end{align*}
\begin{align*}
    \|(\Pi^\perp(\tilde{P}_h-h^2\lambda)\Pi^\perp)^{-1}\|_{H^{s+N+2}\to H^{s+N+2}}\lesssim_{C_0} 1,
    \|hX+h^2\lambda\|_{H^{s+N+2}\to H^{s+N+1}}\lesssim_{C_0}h,
\end{align*}
\begin{align*}
  \| X (c_n\Pi^\perp\Delta_V\Pi^\perp)^{-1}\|_{H^{s+N+1}\to H^{s+N}}\lesssim 1,  \|(\Delta_M-\lambda+Q_A)^{-1}\|_{H^{s+N}\to H^{s+N}}\lesssim_{C_0} A^{-2}.
\end{align*}
We conclude
\begin{align*}
    \|(P_\gamma-\lambda+Q_A)^{-1}Q_A- (\Delta_M-\lambda+Q_A)^{-1}Q_A\|_{H^s\to H^{s+N}}\lesssim_{C_0}A^{N+2}\gamma^{-1}.
\end{align*}
\end{proof}
Now we are ready to prove Theorem \ref{thm:Conv}.
\begin{proof}[Proof of Theorem \ref{thm:Conv}]
We first show the spectrum convergence \eqref{spec_conv}.
It is direct to check $I-(P_\gamma-\lambda+Q_A)^{-1}Q_A$ is invertible on $L^2(SM)$ if and only if it is invertible on $D(P_\gamma)$. So for $U\Subset \CC$,
\begin{align*}
    \sigma(P_\gamma)\cap U&=\{\lambda\in U : I-(P_\gamma-\lambda+Q_A)^{-1}Q_A \text{ is not invertible on }L^2(SM)\}\\
    &=\text{ zeros of }\det( I-(P_\gamma-\lambda+Q_A)^{-1}Q_A ) \text{ in } U.
\end{align*}
By Proposition \ref{prop:conv_QA}, for fixed $A$, the determinant $\det( I-(P_\gamma-\lambda+Q_A)^{-1}Q_A )$ convergences to $\det( I-(\Delta_M-\lambda+Q_A)^{-1}Q_A )$ locally uniformly as $\gamma\to\infty$. So the zeros also converge to zeros of $\det( I-(\Delta_M-\lambda+Q_A)^{-1}Q_A )$, which are exactly eigenvalues of $\Delta_M$.

Now we prove the resolvent convergence \eqref{res_conv}. We will choose $A=\gamma^{1/5}\to\infty$ below. Let $u\in H^s$ with $u_{\rm low}=\mathbbm{1}_{(\Delta_M\leq 10 A^2)}\Pi u$ and $u_{\rm high}=u-u_{\rm low}$, we have
\begin{align*}
    \|(\Delta_M-\lambda)^{-1}u_{\rm high}\|_{H^{s+1/4}}\lesssim_{C_0} A^{-7/4}\|u_{\rm high}\|_{H^s}.
\end{align*}
We note \begin{align*}
    \|(I-Q_A(\Delta_M-\lambda+Q_A)^{-1})^{-1}\|_{H^s\to H^s}=\|I-Q_A(\Delta_M-\lambda)^{-1}\|_{H^s\to H^s}\lesssim_{C_U} 1+A^2
\end{align*}
where $C_U^{-1}$ is the distance between $\sigma(\Delta_M)$ and $U$. For $A^2\gamma^{-1}\ll_U A^{-2}$, we have \begin{align*}
    \|Q_A(P_\gamma-\lambda+Q_A)^{-1}-Q_A(\Delta_M-\lambda+Q_A)^{-1}\|_{H^s\to H^s}\ll \|(I-Q_A(\Delta_M-\lambda+Q_A)^{-1})^{-1}\|_{H^s\to H^s}^{-1}
\end{align*}
so that
\begin{align*}
    \|(I-Q_A(P_\gamma-\lambda+Q_A)^{-1})^{-1}\|_{H^s\to H^s}\lesssim_U A^2.
\end{align*}
Moreover,
\begin{align*}
    &\|((I-Q_A(P_\gamma-\lambda+Q_A)^{-1})^{-1}-(I-Q_A(\Delta_M-\lambda+Q_A)^{-1})^{-1})v\|_{H^s}\\
    &\leq \|(I-Q_A(P_\gamma-\lambda+Q_A)^{-1})^{-1} (Q_A(P_\gamma-\lambda+Q_A)^{-1} -Q_A(\Delta_M-\lambda+Q_A)^{-1}) \\
    &(I-Q_A(\Delta_M-\lambda+Q_A)^{-1})^{-1} v\|_{H^s}\\
    &\lesssim_U A^4\gamma^{-1}\|(I-Q_A(\Delta_M-\lambda+Q_A)^{-1})^{-1}v\|_{H^s}.
\end{align*}
Using \eqref{eq:inv_bound_1/4}, we conclude 
\begin{align*}
     \|(P_\gamma-\lambda)^{-1}u_{\rm high}\|_{H^{s+1/4}}&=\|(P_\gamma-\lambda+Q_A)^{-1}(I-Q_A(P_\gamma-\lambda+Q_A)^{-1})^{-1}u_{\rm high}\|_{H^{s+1/4}}\\
     &\lesssim_{C_0} A^{-1/2}\|(I-Q_A(P_\gamma-\lambda+Q_A)^{-1})^{-1}u_{\rm high}\|_{H^s}\\
     &\lesssim_U A^{-1/2}\|(I-Q_A(\Delta_M-\lambda+Q_A)^{-1})^{-1}u_{\rm high}\|_{H^s}\\
     &=A^{-1/2}\|u_{\rm high}\|_{H^s}.
\end{align*}
In the last step we use the fact $Q_A u_{\rm high}=0$.
Now we are left with the finite dimensional part $u_{\rm low}$ and by Proposition \ref{prop:conv_QA} we have
\begin{align*}
    &\|((P_\gamma-\lambda)^{-1}-(\Delta_M-\lambda)^{-1})u_{\rm low}\|_{H^{s+1/4}}\\
    &=\|((I-(P_\gamma-\lambda+Q_A)^{-1}Q_A)^{-1}(P_\gamma-\lambda+Q_A)^{-1}\\
   & -(I-(\Delta_M-\lambda+Q_A)^{-1}Q_A)^{-1}(\Delta_M-\lambda+Q_A)^{-1})\mathbbm{1}_{(\Delta_M\leq 10 A^2)}\Pi u_{\rm low}\|_{H^{s+1/4}}\\
    &\leq \|(I-(P_\gamma-\lambda+Q_A)^{-1}Q_A)^{-1}((P_\gamma-\lambda+Q_A)^{-1}-(\Delta_M-\lambda+Q_A)^{-1}) u_{\rm low}\|_{H^{s+1/4}}+\\
    &\|((I-(P_\gamma-\lambda+Q_A)^{-1}Q_A)^{-1}-(I-(\Delta_M-\lambda+Q_A)^{-1}Q_A)^{-1})(\Delta_M-\lambda+Q_A)^{-1} u_{\rm low}\|_{H^{s+1/4}}\\
    &\lesssim_U A^{2+1/4}\gamma^{-1}\|u_{\rm low}\|_{H^s}+A^{4+1/4}\gamma^{-1}\|u_{\rm low}\|_{H^s}.
\end{align*}
We conclude
\begin{align}\label{eq:res_conv_quant}
    \|(P_\gamma-\lambda)^{-1}-(\Delta_M-\lambda)^{-1}\|_{H^s\to H^{s+1/4}}\lesssim_U A^{-1/2}+A^{4+1/4}\gamma^{-1}\lesssim \gamma^{-1/10}.
\end{align}
This finishes the proof of \eqref{res_conv}.
\end{proof}

\section{Convergence to equilibrium}\label{sec3}
In this section we give the proof of Theorem \ref{thm:Equili}. In fact, we will prove the following more general Theorem \ref{thm:Expan}. If we take $\beta$ in Theorem \ref{thm:Expan} to be smaller than the first eigenvalue of $\Delta_M$, then there is only a single term coming from the zero eigenvalue of $P_\gamma$ in the expansion, and this gives Theorem \ref{thm:Equili}.

\begin{thm}\label{thm:Expan}
For any $\beta,\epsilon>0$ there exists $\gamma_0>0$ such that for $\gamma>\gamma_0$, $\sigma(P_\gamma)\cap \{\Re\lambda\leq\beta\}=\{\lambda_0=0,\lambda_1,\cdots,\lambda_r\}$ is finite, and there exists $C_\gamma>0$ such that
\begin{align*}
    \left\|e^{-tP_\gamma}u-\sum\limits_{j=0}^r\sum\limits_{l=0}^{m_j-1} \frac{(-t)^l}{l!}e^{-t\lambda_j} (P_\gamma-\lambda_j)^l\Pi_{\lambda_j}u\right\|_{L^2}\leq C_\gamma e^{-\beta t}\|u\|_{L^2},\quad t\geq 1,
\end{align*}
where $\Pi_{\lambda_j}$ is the spectral projector to the generalized eigenspace of $P_\gamma$ with eigenvalue $\lambda_j$. Moreover, for each $j$ there is $\lambda_j^0\in \sigma(\Delta_M)$ such that $|\lambda_j-\lambda_j^0|<\epsilon$.
\end{thm}
\begin{proof}
First we claim there are only finitely many eigenvalues of $P_\gamma$ in the region $\{\Re\lambda\leq\beta\}$, and they all satisfy $|\Im\lambda|\lesssim\beta$. We prove by contradiction again. Suppose $\lambda$ is an eigenvalue of $P_\gamma$ such that $\Re\lambda\leq \beta$, then there exists $u\in C^\infty(SM)$ such that
\begin{align}\label{eqn:eigen_stripe}
    P_\gamma u= c_n \gamma^2\Delta_V u-\gamma Xu =\lambda u.
\end{align}
As in the proof of Lemma \ref{inv_lem_1}, we have
\begin{align*}
    \|\Pi^\perp u\|_{H^s}\lesssim \sqrt{\beta} \gamma^{-1}\|u_0\|_{H^s},\quad s=0,1.
\end{align*}
Projecting the equation \eqref{eqn:eigen_stripe} to $V_0$ gives 
\begin{align*}
    -\gamma\Pi_0 X u_1 =\lambda u_0.
\end{align*}
Thus $|\lambda| \|u_0\|_{L^2}\lesssim \gamma\|u_1\|_{H^1}\lesssim \sqrt{\beta}\|u_0\|_{H^1}$. Projecting the equation \eqref{eqn:eigen_stripe} to $V_1$ gives
\begin{align*}
    \frac{1}{n}\gamma^2 u_1 -\gamma \Pi_1 X(u_0+u_2) =\lambda u_1
\end{align*}
which gives as before $\|u_0\|_{H^1}\lesssim \sqrt{\beta}(1+\gamma^{-2}|\lambda|)\|u_0\|_{L^2}$. We conclude
\begin{align*}
     |\lambda|\lesssim \beta (1+\gamma^{-2}|\lambda|).
\end{align*}
Taking $\gamma^2\gg \beta$, we conclude $|\lambda|\lesssim \beta$. Along with Theorem \ref{thm:Conv}, this shows $|\lambda_j-\lambda_j^0|<\epsilon$ for some $\lambda_j^0\in\sigma(\Delta_M)$ once $\gamma>\gamma_0$ is taken large enough.

Now we consider the Laplace transform of $e^{-tP_\gamma}$:
\begin{align*}
    \int_0^\infty e^{\lambda t}e^{-tP_\gamma} dt =(P_\gamma-\lambda)^{-1},\quad \Re\lambda<0.
\end{align*}
We can then express $e^{-tP_\gamma}$ as the inverse Laplace transform
\begin{align*}
    e^{-tP_\gamma}=\frac{1}{2\pi i}\int_{-1-i\infty}^{-1+i\infty}(P_\gamma-\lambda)^{-1}e^{-\lambda t}d\lambda.
\end{align*}
We deform the contour from $\Re\lambda=-1$ to $\rho$ and conclude
\begin{align}\label{eq:contour}
    e^{-tP_\gamma}=\sum\limits_{j=0}^r {\rm Res}_{\lambda=\lambda_j}((\lambda-P_\gamma)^{-1}e^{-\lambda t})+\frac{1}{2\pi i}\int_\rho (P_\gamma-\lambda)^{-1} e^{-\lambda t}d\lambda
\end{align}
where $\rho$ is given by $\{\Re\lambda=\beta+\epsilon_0,|\Im\lambda|\leq A_\gamma\}$ and $\{|\Im\lambda|= C_\gamma(\Re\lambda-\beta-\epsilon_0)^{16}+A_\gamma,\Re\lambda>\beta+\epsilon_0\}$. See Figure \ref{contour} for a picture of the contours.
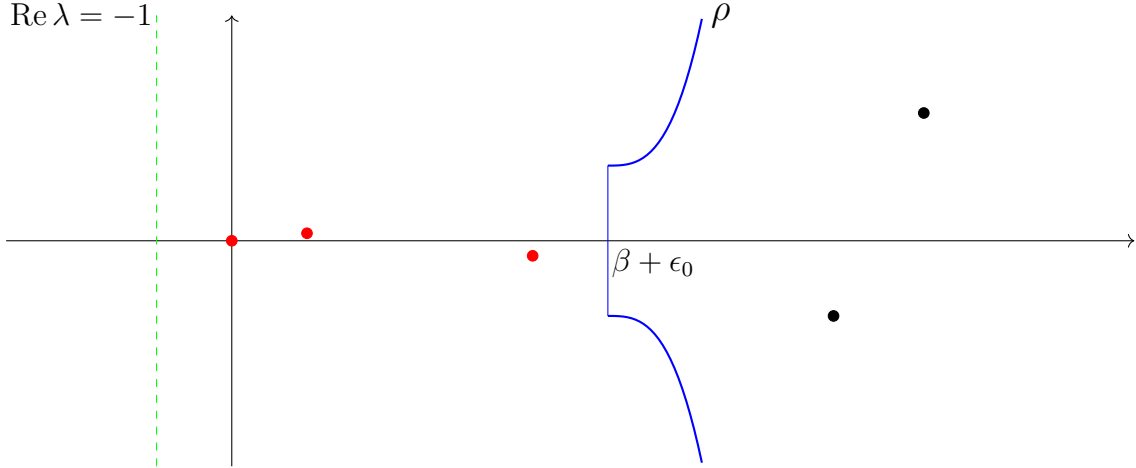
\begin{figure}
    \centering
\begin{tikzpicture}
\draw[->] (-3,0) -- (12,0);
\draw[->] (0,-3) -- (0,3);
\draw[dashed, color=green] (-1,-3) -- (-1,3);
\draw [color=blue] (5,-1) -- (5,1); 
\draw [blue,thick,domain=0:1.25] plot ({\x+5}, {1+\x^3});
\draw [blue,thick,domain=0:1.25] plot ({\x+5}, {-1-\x^3});
\filldraw[red] (0,0) circle (2pt);
\filldraw[red] (1,0.1) circle (2pt);
\filldraw[red] (4,-0.2) circle (2pt);
\filldraw[black] (8,-1) circle (2pt);
\filldraw[black] (9.2,1.7) circle (2pt);
\node at (-2,3) {$\Re\lambda=-1$};
\node at (6.5,3) {\large$\rho$};
\node at (5.6,-0.3) {$\beta+\epsilon_0$};
\end{tikzpicture}
\caption{Contour deformation}
\label{contour}
\end{figure}
In order to conclude the proof we need the following Lemma \ref{lem:e-h} from Eckmann--Hairer \cite[Theorem 4.1, 4.3]{eckmann2003spectral}.
\begin{lem}\label{lem:e-h}
There exists $C>0$ (independent of $\gamma$) such that $P_\gamma$ does not have spectrum in $\{|\Im\lambda|\geq C(\Re\lambda+\gamma^{1/4})^{16}+1, \Re\lambda>0\}$. Moreover, we have for such $\lambda$
\begin{align*}
    \|(P_\gamma-\lambda)^{-1}\|_{L^2\to L^2}\lesssim 1.
\end{align*}
\end{lem}
\begin{proof}
The lemma follows from the uniform hypoelliptic estimate \eqref{eq:unif_hypo_iy}
\begin{align*}
    \| u\|_{H^{1/8}}\leq C(B^{-1}\|(P_\gamma-iy) u\|_{L^2}+B\|u\|_{L^2}),\qquad \forall y\in\RR
\end{align*}
with constant $C>0$ independent of $y$ and $\gamma$. Taking $B=\gamma^{1/8}$, we get (using \cite[[Proposition B.1]{herau2004isotropic})
\begin{align*}
    \frac{1}{4}|\lambda+1|^{1/8}\|u\|_{L^2}^2&\leq (((P_\gamma+1)^*(P_\gamma+1))^{1/16}u,u)_{L^2}+\|(P_\gamma-\lambda)u\|_{L^2}^2\\
    &\lesssim  \gamma^{1/4}\|u\|_{H^{1/8}}^2+\|(P_\gamma-\lambda)u\|_{L^2}^2\\
    &\lesssim (\gamma^{1/4}+\Re\lambda)^2\|u\|_{L^2}^2+ \|(P_\gamma-\lambda) u\|_{L^2}^2.
\end{align*} Thus for $|\lambda+1|\geq C_1 (\gamma^{1/4}+\Re\lambda)^{16}+1$ we conclude
\[\|u\|_{L^2}\lesssim \|(P_\gamma-\lambda)u\|_{L^2}.\qedhere\]
\end{proof}
Theorem \ref{thm:Expan} then follows from \eqref{eq:contour} where the residues are given by
\begin{align*}
    {\rm Res}_{\lambda=\lambda_j}((\lambda-P_\gamma)^{-1}e^{-\lambda t})&={\rm Res}_{\lambda=\lambda_j}\left(\sum\limits_{l=0}^{m_j-1}(P_\gamma-\lambda_j)^l\Pi_{\lambda_j}(\lambda-\lambda_j)^{-l-1} e^{-\lambda t}\right)\\
    &=\sum\limits_{l=0}^{m_j-1}\frac{(-t)^l}{l!} e^{-\lambda_j t}(P_\gamma-\lambda_j)^l\Pi_{\lambda_j}
\end{align*}
and the remainder is estimated as
\[\left\|\int_\rho (P_\gamma-\lambda)^{-1} e^{-\lambda t}d\lambda\right\|_{L^2\to L^2}\lesssim_\gamma \int_\rho  e^{-\Re\lambda t}d|\lambda|\lesssim_\gamma e^{-\beta t}.\qedhere\]
\end{proof}







\printbibliography








\end{document}